\documentclass[english]{amsart}

\usepackage{amsmath}
\usepackage{amssymb}
\usepackage{amsthm}
\usepackage{amscd}
\usepackage{babel}
\usepackage[latin1]{inputenc}
\usepackage{graphicx}

\newtheorem{teor}{Theorem}
\newtheorem{cor}{Corollary}
\newtheorem{prop}{Proposition}
\newtheorem{con}{Conjecture}
\newtheorem{lem}{Lemma}
\newtheorem{defi}{Definition}

\theoremstyle{definition}
\newtheorem*{rem}{Remark}

\renewcommand{\subjclassname}{AMS \textup{2010} Mathematics Subject
Classification\ }

\author{Jos\'{e} Mar\'{i}a Grau}
\address{Departamento de Matemáticas, Universidad de Oviedo\\ Avda. Calvo Sotelo s/n, 33007 Oviedo, Spain}
\email{grau@uniovi.es}

\author{Antonio M. Oller-Marc\'{e}n}
\address{Centro Universitario de la Defensa\\ Ctra. Huesca s/n, 50090 Zaragoza, Spain} \email{oller@unizar.es}
\title{Generalizing Giuga's conjecture}

\title{On the congruence $\sum_{j=1}^{n-1}  j^{k(n-1)} \equiv -1 \pmod n $. k-strong Giuga and k-Carmichael numbers}

\begin{document}

\begin{abstract}
In this work we consider the congruence $\sum_{j=1}^{n-1}  j^{k(n-1)} \equiv -1 \pmod n$ for each $k \in \mathbb{N}$, thus extending Giuga's ideas for $k=1$. In particular, it is proved that a pair $(n,k)\in \mathbb{N}^2$  satisfies this congruence if and only if $n$ is prime or a Giuga Number and $\lambda(n) \mid k(n-1)$. In passing, we establish new characterizations of Giuga numbers and we study some properties of the numbers $n$ satisfying $\lambda(n) \mid k(n-1)$.
\end{abstract}

\maketitle
\subjclassname{11B99, 11A99, 11A07}


\section{Introduction}
In 1950, G. Giuga conjectured \cite{GIU} that if an integer $n$ satisfies ${\sum_{j=1}^{n-1}  j^{n-1} \equiv -1}$ (mod $n$), then $n$ must be a prime. Moreover, Giuga proved that $n$ is a counterexample to his conjecture if and only if for each prime divisor $p$ of $n$, $p(p-1)\mid (n/p-1)$. In what follows, counterexamples to Giuga's conjecture will be called \emph{strong Giuga numbers}.

Using the above characterization, Giuga proved computationally that any strong Giuga number has at least 1000 digits. Equipped with more computing power, Bedocchi \cite{BEDO} later raised this bound to 1700 digits. Improving their method, D. Borwein, J. M. Borwein, P. B. Borwein and R. Girgensohn \cite{BOR} determined that any strong Giuga number contains at least 3459 distinct primes and so has at least 13887 digits.

On the other hand, Luca, Pomerance and Shparlinski \cite{LUCA} have established the following bound (which improves previous work by Tipu \cite{TIPU}) for the counting function of the strong Giuga numbers:
$$\left| \{n<X :  n \textrm{ is a strong Giuga numbers }\}\right|  \ll \frac{X^{\frac{1}{2}}}{(\log(X))^2}.$$

Kellner  has stablished \cite{KELL} that Giuga's conjecture is equivalent to the following conjecture by Agoh \cite{AGO}.

\begin{con}[Agoh, 1995]
Let $B_k$ denote the k-th Bernoulli number. Then, $n$ is a prime if and only if $nB_{n-1} \equiv -1  \pmod n.$
\end{con}

Borwein et al. \cite{BOR} introduced the following definition of \emph{Giuga numbers} relajando la propiedad de divisibilidad of the strong Giuga numbers.

\begin{defi}
A Giuga Number is a composite number $n$ such that $p\mid (n/p-1)$ for every prime divisor $p$ of $n$.
\end{defi}
Con esta definición pueden caracterizarse los strong Giuga number de la siguiente manera:
\begin{prop}
Let $n$ be a composite integer. Then $n$ is a strong Giuga number if and only if $n$ is both a Giuga number and a Carmichael number.
\end{prop}

There are several equivalent definitions of Giuga numbers. Some of them are the following.

\begin{prop}
Let $n$ be a composite integer. Then, the following are equivalent:
\begin{itemize}
\item [i)] $n$ is a Giuga number.
\item [ii)] Giuga \cite{GIU}: $\displaystyle{\sum_{p\mid n} \frac{1}{p}-\frac{1}{n} \in \mathbb{N}}$.
\item[iii)] Borwein et al. \cite{BOR}: $\displaystyle{\sum_{j=1}^{n-1}  j^{\phi(n)} \equiv -1} \pmod n$, where $\phi$ is Euler's totient function.
\item [iv)] Agoh \cite{AGO}: $nB_{\phi(n)} \equiv -1  \pmod n$, where $B$ is a Bernoulli number.
\item[v)] Grau and Oller \cite{GOL}: $n'=an+1$ for some $a\in \mathbb{N}$, where $n'$ is the arithmetic derivative of $n$.
\end{itemize}
\end{prop}

Up to date only thirteen Giuga numbers are known (see A007850 in the \emph{On-Line Encyclopedia of Integer Sequences}):
\begin{align*}
&\mathfrak{g}_1:=\textbf{30},\\
&\mathfrak{g}_2:=\textbf{858},\\
&\mathfrak{g}_3:=\textbf{1722},\\
&\mathfrak{g}_4:=\textbf{66198},\\
&\mathfrak{g}_5:=\textbf{2214408306},\\
&\mathfrak{g}_6:=\textbf{24423128562},\\
&\mathfrak{g}_7:=\textbf{432749205173838},\\
&\mathfrak{g}_8:=\textbf{14737133470010574},\\
&\mathfrak{g}_9:=\textbf{550843391309130318},\\
&\mathfrak{g}_{10}:=\textbf{244197000982499715087866346},\\
&\mathfrak{g}_{11}:=\textbf{554079914617070801288578559178},\\
&\mathfrak{g}_{12}:=\textbf{1910667181420507984555759916338506},\\
&\mathfrak{g}_{13}:=\textbf{420001794970774706203871150967065663240419575375163060922}\\ &\textbf{8764416142557211582098432545190323474818}.
\end{align*}

All known Giuga numbers are even and it is known that if an odd Giuga number exists, it must be the product of at least 14 primes. It is not even known if there are infinitely many Giuga numbers.

Variants of Giuga numbers have already been proposed by the authors \cite{GLO}. In this work some new characterizations of Giuga numbers are stablished. This characterizations arise from the study of the congruence $\sum_{j=1}^{n-1}  j^{k(n-1)} \equiv -1$ (mod $n$) with $k \in \mathbb{N}$. It is proved that a pair $(n,k)\in \mathbb{N}^2$, with composite $n$, satisfies this congruence if and only if $n$ is a Giuga number and $ \lambda(n)$ divides $k(n-1)$. This last property leads to a generalization of Carmichael numbers (the  $k$-\emph{Carmichael numbers}) which are also characterized in the square-free case.

\section{New characterizations of Giuga numbers}

The following result establishes that in Proposition 2 iii) we can replace Euler's totient function $\phi(n)$ by Carmichael's function $\lambda(n)$ or by any multiple of $\phi(n)$ or $\lambda(n)$.

\begin{lem}
For every natural numbers $A$, $B$ and $N$ we have that:
$$\sum_{j=1}^{N-1}  j^{A \lambda(N)} \equiv  \sum_{j=1}^{N-1}  j^{B \phi(N)}\ \textrm{(mod $N$)}.$$
\end{lem}
\begin{proof}
Put $N=2^ap_1^{r_1}\cdots p_s^{r_s}$ with $p_i$ distinct odd primes. Choose $i\in \{1,\dots, s\}$. We have that:
$$\sum_{j=1}^{N-1} j^{A\lambda(N)} \equiv \frac{N}{p_i^{r_i}}\sum_{j=1}^{p_i^{r_i}-1} j^{A\lambda(N)}\ \textrm{(mod $p_i^{r_i}$)}.$$
$$\sum_{j=1}^{N-1} j^{B\phi(N)} \equiv \frac{N}{p_i^{r_i}}\sum_{j=1}^{p_i^{r_i}-1} j^{B\phi(N)}\ \textrm{(mod $p_i^{r_i}$)}.$$
Now, since both $A\lambda(N),\ B\phi(N)\geq r_i$, we get:
$$\sum_{j=1}^{p_i^{r_i}-1} j^{A\lambda(N)}=\sum_{\substack{1\leq j\leq p_i^{r_i}-1\\ (p_i,j)=1}} j^{A\lambda(N)}+\sum_{\substack{1\leq j\leq p_i^{r_i}-1\\ p_i\mid j}} j^{A\lambda(N)}\equiv \phi(p_i^{r_i})+0\ \textrm{(mod ($p_i^{r_i}$)}.$$
$$\sum_{j=1}^{p_i^{r_i}-1} j^{B\phi(N)}=\sum_{\substack{1\leq j\leq p_i^{r_i}-1\\ (p_i,j)=1}} j^{B\phi(N)}+\sum_{\substack{1\leq j\leq p_i^{r_i}-1\\ p_i\mid j}} j^{A\lambda(N)}\equiv \phi(p_i^{r_i})+0\ \textrm{(mod ($p_i^{r_i}$)}.$$
Consequently:
$$\sum_{j=1}^{N-1} j^{A\lambda(N)}\equiv\sum_{j=1}^{N-1} j^{B\phi(N)}\ \textrm{(mod $p_i^{r_i}$)}\ \textrm{for every $i=1,\dots,s$}.$$
Clearly if $N$ is odd the proof is complete. If $n$ is even we have that:
$$\sum_{j=1}^{N-1}j^{A\lambda(N)}\equiv\frac{N}{2^a}\sum_{j=1}^{2^a-1} j^{A\lambda(N)}\equiv\frac{N}{2^a}\left(\sum_{\substack{1\leq j\leq 2^a-1\\ \textrm{$j$ even}}} j^{A\lambda(N)} + 2^{a-1}\right)\ \textrm{(mod $2^a$)}.$$
$$\sum_{j=1}^{N-1}j^{A\lambda(N)}\equiv\frac{N}{2^a}\left(\sum_{\substack{1\leq j\leq 2^a-1\\ \textrm{$j$ even}}} j^{B\phi(N)} + 2^{a-1}\right)\ \textrm{(mod $2^a$)}.$$
Now, if $a=1,2$ or $3$ it can be easily verified that:
$$\sum_{\substack{1\leq j\leq 2^a-1\\ \textrm{$j$ even}}} j^{A\lambda(N)}\equiv\sum_{\substack{1\leq j\leq 2^a-1\\ \textrm{$j$ even}}} j^{B\phi(N)}\ \textrm{(mod $2^a$)}.$$
On the other hand, if $a\geq 4$ we have that $\phi(N)\geq\lambda(N)\geq a$ and, consequently $j^{A\lambda(N)}\equiv j^{B\phi(N)}\equiv 0$ (mod $2^a$) for every $1\leq j\leq 2^{a-1}$ even. Thus:
$$\sum_{j=1}^{N-1} j^{A\lambda(N)}\equiv\sum_{j=1}^{N-1} j^{B\phi(N)}\ \textrm{(mod $2^a$)}$$
and the result follows.
\end{proof}

This lemma leads to some new characterizations of Giuga numbers. Recall that a composite integer $n$ is said to be a Giuga number (among other characterizations) if and only if ${\sum_{j=1}^{n-1} j^{\phi(n)}\equiv -1}$ (mod $n$).

\begin{prop}
Let $n$ be any composite integer. Then the following are equivalent:
\begin{itemize}
\item[i)] $n$ is a Giuga Number.
\item[ii)] For every positive integer $K$, $\displaystyle{\sum_{j=1}^{n-1}  j^{K\lambda(n)}\equiv \sum_{j=1}^{n-1}  j^{K\phi(n)}} \equiv -1$ (mod $n$).
\item [iii)] There exists a positive integer $K$ such that $\displaystyle{\sum_{j=1}^{n-1}  j^{K\lambda(n)}\equiv \sum_{j=1}^{n-1}  j^{K\phi(n)}} \equiv -1$ (mod $n$).
\end{itemize}
\end{prop}




The following result, which is a consequence of previous proposition, will allow us to generalize Giuga's ideas by considering the congruence ${\sum_{j=1}^{n-1}  j^{k(n-1)} \equiv  -1}$ (mod $n$) for each positive integer $k$.

\begin{cor}
If an integer $n$ is a strong Giuga number, then:
$$\sum_{j=1}^{n-1}  j^{k(n-1)} \equiv  -1\ \textrm{(mod $n$) for every positive integer $k$}.$$
\end{cor}
\begin{proof}
If $n$ is a strong Giuga number, then it is both a Carmichael and a Giuga number. Being a Carmichael number, we have that $\lambda(n) \mid  (n-1)$ so if $\frac{k(n-1)}{\lambda(n)}=k'\in\mathbb{N}$ we get:
$$S:=\sum_{j=1}^{n-1}  j^{k(n-1)}=\sum_{j=1}^{n-1}  j^{k \lambda(n)\frac{(n-1)}{\lambda(n)}}=\sum_{j=1}^{n-1}  j^{k' \lambda(n)},$$
and, since $n$ is a Giuga number it is enough to apply Corollary 1 and Proposition 3 to get $S\equiv  -1$ (mod $n$).
\end{proof}

\section{$k$-Carmichael numbers}

A Carmichael number is a composite positive integer $n$ which satisfies the congruence $a^{n-1}\equiv 1$ (mod $n$) for every integer $a$ coprime to $n$. Korselt \cite{KOR} was the first to observe the basic properties of Carmichael numbers, the most important being the following characterization:

\begin{prop}[Korselt, 1899]
A composite number $n$ is a Carmichael number if and only if $n$ is square-free, and for each prime $p$ dividing $n$, $p-1$ divides $n-1$.
\end{prop}

Nevertheless, Korselt did not find any example and it was Carmichael \cite{CAR2} who found the first and smallest of such numbers (561) and hence the name ``Carmichael number'' (which was introduced by Beeger \cite{BEE}). In the same paper Carmichael presents the following characterization in terms of a reduced totient function  $\lambda$.

\begin{prop}[Carmichael, 1912]
A composite number $n$ is a Carmichael number if and only if $\lambda(n)$ divides $(n-1)$.
\end{prop}

Motivated by this characterization we have the following definition.

\begin{defi}
Given $k\in \mathbb{N}$ we will say that $n$ is a $k$-Carmichael number if $\lambda(n)$ divides $k(n-1)$.
\end{defi}

Square-free $k$-Carmichael numbers admit, at least, the following characterizations.

\begin{prop}
Let $n$ be a square-free positive composite integer and let $k\in \mathbb{N}$. The following are equivalent:
\begin{itemize}
\item[i)] $n$ is a $k$-Carmichael number.
\item[ii)] For every  prime divisor $p$ of $n$, $p-1$ divides $k(n-1)$.
\item[iii)] For every integer $a$, $a^{kn}\equiv a^k$ (mod $n$).
\end{itemize}
\end{prop}
\begin{proof}
If $n$ is a square-free $k$-Carmichael number, then $\lambda(n)=\textrm{lcm}\{p-1\ |\ p\ \textrm{divides}\ n\}$ divides $k(n-1)$. This proves that i) implies ii).

Now, if $p-1$ divides $k(n-1)$ for every prime divisor $p$ of $n$ and given any integer $a$, it follows that $a^{k(n-1)}\equiv 1$ (mod $p$) for every prime divisor $p$ of $n$ such that $p$ does not divide $a$. If $p$ divides $a$ the same congruence follows trivially and this proves that ii) implies iii).

Finally, to see that iii) implies i) it is enough to consider an integer $a$ coprime to $n$.
\end{proof}

\begin{rem}
Observe that square-free $k$-Carmichael numbers are Fermat pseudoprimes to base $a^k$ for every $a$ such that $\gcd(a,n)=1$.
\end{rem}

\begin{rem}
If $n$ is a $k$-Carmichael number and $p^2 \mid n$, then we trivially have that $p$ divides $k$.
\end{rem}

We close this section with some conjectures regarding $k$-Carmichael numbers. We consider $\mathfrak{C}_k:=\{n :\textrm{ $n$ is a $k$-Carmichael number }\}$. It is clear that these sets satisfy that $k \mid s \Rightarrow \mathfrak{C}_s \subset \mathfrak{C}_s $. In particular the set of Carmichael numbers, $\mathfrak{C}_1$, is contained in any other $\mathfrak{C}_s$. Although it is not the point of this work, it could be interesting to study the relative asympotitc density of $\mathfrak{C}_k$ with respect to $\mathfrak{C}_s$ for every $s$ multiple of $k$. In particular the density of $\mathfrak{C}_1$ with respect to $\mathfrak{C}_k$  $$\delta_k:=\lim_{n\rightarrow\infty}\frac{|\mathfrak{C}_1 \bigcap [1,n]|}{|\mathfrak{C}_k \bigcap [1,n]|}.$$ 
We conjecture that it exists and that it is positive and strictly smaller than 1.

\begin{con} 
If $k \mid t$ then $\mathfrak{C}_k$ has a relative asymptotic density with respect to $\mathfrak{C}_t$. Moreover, for every $k,t \in \mathbb{N}$ $$\lim_{n\rightarrow\infty}\frac{|\mathfrak{C}_k \bigcap [1,n]|}{|\mathfrak{C}_t \bigcap [1,n]|}=\frac{\delta_t}{\delta_k}. $$
\end{con}

If this conjecture was true, three interesting questions arise:
\begin{itemize}
\item[i)] Given integers $m$ and $n$ such that none of them divides the other, which is the case: $\delta_m < \delta_n$, $\delta_m > \delta_n$ or $ \delta_m = \delta_n $?
\item[ii)] Does it exist a prime $\mathfrak{p}$ such that $\delta_\mathfrak{p}<\delta_p$ for all prime $p \in \mathbb{\mathbb{N}}-\{\mathfrak{p}\}$? Soon we will conjeture its existence and that $\mathfrak{p}= 5$.
\item[iii)] which is the smallest composite integer $\mathfrak{c}$ such that $\delta_{\mathfrak{c}}> \delta_5$.
\end{itemize}

As we said we conjecture the following.

\begin{con} 
$0.18<\delta_5<\delta_p$ for all  prime $p \in \mathbb{\mathbb{N}}-\{5\}$.
\end{con}
 
Of course all this could just be a computational mirage and the truth is just that ${\delta_k}=1$ for all $k\in \mathbb{N} $ but it seems worthwhile to find it out. The values of the counter function  $\mathcal{C}_k(n):= |\mathfrak{C}_k \bigcap [1,n]|$ for $ n \in \{10^6,10^7\}$ can be found in A231575 and A231574 from the OEIS. The relative order of $\mathcal{C}_k(n)$ for different values of $k$ changes when we increase $n$, nevertheless it seems that for every prime $p$ and $n>10^5$: 
$$\mathcal{C}_5(n)>\mathcal{C}_3(n)>\mathcal{C}_7(n)>\mathcal{C}_p(n)$$.

\section{$k$-strong Giuga numbers}

In this section we extend Giuga's ideas studying pairs of integers $k$ and $n$ such that ${\sum_{j=1}^{n-1}  j^{k(n-1)} \equiv  -1}$ (mod $n$). This motivates the following definition.

\begin{defi}
Given $k\in \mathbb{N}$ we say that a composite number $n$ is a $k$-strong Giuga number if
$$\sum_{j=1}^{n-1}  j^{k(n-1)} \equiv  -1 (\textrm{mod }n).$$
\end{defi}

We can also define the following sets:
$$\mathcal{G}_k:=\{ n\in \mathbb{N}\ |\ n \textrm{ is $k$-strong Giuga number} \},$$
$$\mathcal{K}_n:=\{k\in \mathbb{N}\ |\ n \textrm{ is $k$-strong Giuga number} \}.$$

\begin{rem}
With the previous notation, Giuga's conjecture is equivalent to the statement $\mathcal{G}_1=\emptyset$.
\end{rem}
A continuación tenemos el principal resultado de este trabajo con el que se caracteriza the k-strong Giuga numbers.
\begin{teor}
Let $n$ be a composite integer. Then $n$ is a $k$-strong Giuga number if and only if $n$ is both a Giuga number and a $k$-Carmichael number.
\end{teor}
\begin{proof}
Assume that $n$ is a Giuga number and a $k$-Carmichael number. Since $\lambda(n)$ divides $k(n-1)$ we have that:
$$\sum_{j=1}^{n-1}  j^{k(n-1)} = \sum_{j=1}^{n-1}  j^{k'\lambda(n)}\equiv -1$$
due to Corollary 1.

Conversely, assume that $n$ is a $k$-strong Giuga number; i.e., ${\sum_{j=1}^{n-1}j^{k(n-1)}\equiv -1}$ (mod $n$). As a consequence (see \cite[Theorem 2.3]{WONG}) we have that $p-1$ divides $k\left(n/p-1\right)$, that $p$ divides $n/p-1$ for every $p$, prime divisor of $n$ and, moreover, that $n$ is square-free. Since $n$ is square-free, $\lambda(n)=\textrm{lcm}\{p-1\ | \ \textrm{$p$ prime dividing $n$}\}$. Thus, $\lambda(n)$ divides $k(n-1)$ and $n$ is a $k$-Carmichael number. To get that $n$ is also a Giuga number, due to Proposition 1, it is enough to apply Proposition 2 with $B=1$ and $A=\frac{k(n-1)}{\lambda(n)}$.
\end{proof}

In the particular case of $ k = 1 $, we have the characterization  that was given in proposition 1 of the strong Giuga numbers.

Taking into account now that the condition $\lambda(n)$ divides $k(n-1)$ is equivalent to $\frac{\lambda(n)}{\gcd(\lambda(n),n-1)}$ divides $k$, given any positive integer $n$, Theorem 1 gives a complete description of the set $\mathcal{K}_n$ as stated in the following corollary.

\begin{cor}
Let $n$ be any composite positive integer. Then:
$$\mathcal{K}_n =\begin{cases}  \left\{t\cdot\frac{\lambda(n)}{gcd(\lambda(n),n-1)}\ | \ t\in \mathbb{N} \right\}, & \textrm{if $n$ is a Giuga number;}\\ \emptyset, & \textrm{otherwise.}\\ \end{cases}$$
\end{cor}

\begin{rem}
Note that $\mathcal{K}_n=\mathbb{N}$ if and only if $n$ is a prime or a strong Giuga number.
\end{rem}

Since it is easy to see that $n\in\mathcal{G}_k$ if and only if $k\in\mathcal{K}_n$ we also have the following result.

\begin{cor}
$\mathcal{G}_k$ is nonempty if and only if  $\lambda(n)$ divides $k(n-1)$ for some Giuga number $n$.
\end{cor}

Using this last result we can find values of $k$ such that $\mathcal{G}_k$ is nonempty. To do so, we evaluate $k(n):=\frac{\lambda(n)}{\gcd(\lambda(n),n-1)}$ for every known Giuga number. Thus, we will have thirteen values of $k$ (recall the introduction) for which $\mathcal{G}_{tk}$ is known to be nonempty for any $t$:
\begin{align*}
k({\mathfrak{g}_1})= & 4;\\
k({\mathfrak{g}_2})= & 60;\\
k({\mathfrak{g}_3})= & 120;\\
k({\mathfrak{g}_4})= &2320;\\
k({\mathfrak{g}_5})=&1552848;\\
k({\mathfrak{g}_6})=&10080;\\
k({\mathfrak{g}_7})=&139714902540;\\
k({\mathfrak{g}_8})=&93294624780;\\
k({\mathfrak{g}_9})=&228657996794220;\\
k({\mathfrak{g}_{10}})=&4756736241732916394976;\\
k({\mathfrak{g}_{11}})=&20024071474861042488900;\\
k({\mathfrak{g}_{12}})=&2176937111336664570375832140;\\
k({\mathfrak{g}_{13}})=&15366743578393906356665002406454800354974137359272\\&445859047945613961394951904884493965220.
\end{align*}

For these values and $t=1$ one gets:

\begin{align*}
\mathcal{G}_{k({\mathfrak{g}_1})}&=\{\mathfrak{g}_1,\dots\};\\
\mathcal{G}_{k({\mathfrak{g}_2})}&=\{\mathfrak{g}_1, \mathfrak{g}_2,\dots\};\\
\mathcal{G}_{k({\mathfrak{g}_3})}&=\{\mathfrak{g}_1, \mathfrak{g}_2, \mathfrak{g}_3,\dots\};\\
\mathcal{G}_{k({\mathfrak{g}_4})}&=\{\mathfrak{g}_1, \mathfrak{g}_4,\dots\};\\
\mathcal{G}_{k({\mathfrak{g}_5})}&=\{\mathfrak{g}_1,\mathfrak{g}_5,\dots\};\\
\mathcal{G}_{k({\mathfrak{g}_6})}&=\{\mathfrak{g}_1,\mathfrak{g}_2, \mathfrak{g}_6,\dots\};\\
\mathcal{G}_{k({\mathfrak{g}_7})}&=\{\mathfrak{g}_1,\mathfrak{g}_7,\dots\};\\
\mathcal{G}_{k({\mathfrak{g}_8})}&=\{\mathfrak{g}_1,\mathfrak{g}_2,\mathfrak{g}_8,\dots\};\\
\mathcal{G}_{k({\mathfrak{g}_9})}&=\{\mathfrak{g}_1,\mathfrak{g}_2,\mathfrak{g}_9,\dots\};\\
\mathcal{G}_{k({\mathfrak{g}_{10}})}&=\{\mathfrak{g}_1,\mathfrak{g}_{10},\dots\};\\
\mathcal{G}_{k({\mathfrak{g}_{11}})}&=\{\mathfrak{g}_1,\mathfrak{g}_{11},\dots\};\\
\mathcal{G}_{k({\mathfrak{g}_{12}})}&=\{\mathfrak{g}_1,\mathfrak{g}_2,\mathfrak{g}_{12},\dots\};\\
\mathcal{G}_{k({\mathfrak{g}_{13}})}&=\{\mathfrak{g}_1,\mathfrak{g}_2,\mathfrak{g}_{13},\dots\}.
\end{align*}

\section{Reflecting about Giuga's conjecture }

In recent work by W. D. Banks, C. W. Nevans, and C. Pomerance \cite{POM}, the following bounds were given:

\begin{teor} For any fixed $ \varepsilon> 0$, $\beta=0.3322408$ and all sufficiently large X, we have
$$\mid  \{n<X :  n \in \mathfrak{C}_1 \} \mid  \geq X^{\beta-\varepsilon} \textrm{ (G. Harman \cite{HAR})}$$
$$\mid  \{n<X : n \in \mathfrak{C}_1 \backslash \mathcal{G}_1 \} \mid  \geq X^{\beta-\varepsilon} $$
\end{teor}

The authors of the aforementioned paper consider the above bounds to be ``consistent''  with Giuga's conjecture. We believe, however, that the similar consideration could be made with respect to conjectures that are actually false, such as  $\mathcal{G}_4=\emptyset$ or $\mathcal{G}_{1552848}=\emptyset$. 
The authors of the present paper, in view of the generalization presented here, are convinced that Giuga's conjecture is not based on any sound logical-mathematical consideration and that its strength rests only on the extreme rarity of Giuga numbers, combined with the 
null asymptotic density of Carmichael numbers. In fact, if we may be forgiven the joke, we might conjecture -without any fear of our conjecture being refuted in many years- that $\mathcal{G}_2=\mathcal{G}_3=\mathcal{G}_5=\emptyset$ or, to be even more daring, that  $\mathcal{G}_p=\emptyset$ for all prime $p$. Of course, Giuga's conjecture has the honour of being the strongest of all of these conjectures. In fact, in virtue of Corollary 1, should it be refuted, all the others would fall with it.

\end{document}